\documentclass{amsart}[12pt]
\usepackage{amssymb,amsmath}
\usepackage{enumerate}
\usepackage[english]{babel}
\usepackage{color}

\newtheorem{thm}{Theorem}
\newtheorem{lem}[thm]{Lemma}

\newtheorem{cor}[thm]{Corollary}
\newtheorem{property}[thm]{Property}
\newtheorem{prob}[thm]{Problem}
\begin{document}

\title{Fuglede-Putnam theorem for locally measurable operators}

\author{A. Ber}
\address{National University of Uzbekistan, Tashkent, Uzbekistan}
\email{aber1960@mail.ru}
\author{V. Chilin}
\address{National University of Uzbekistan, Tashkent, Uzbekistan}
\email{chilin@ucd.uz}
\author{F. Sukochev}
\address{School of Mathematics and Statistics, University of New South Wales, Kensington, 2052, Australia}
\email{f.sukochev@unsw.edu.au}
\author{D. Zanin}
\address{School of Mathematics and Statistics, University of New South Wales, Kensington, 2052, Australia}
\email{d.zanin@unsw.edu.au}

\begin{abstract}
We extend the Fuglede-Putnam theorem from the algebra $B(H)$ of all bounded operators on the Hilbert space $H$ to the algebra of all locally measurable operators affiliated with a von Neumann algebra.
\end{abstract}


\maketitle

\section{Introduction}

The (first part of the) following problem was suggested by von Neumann (see pp. 60-61, Appendix 3 in \cite{N}).

\begin{prob}\label{von neumann problem} Let $a,b,c\in B(H).$ If $a$ is normal and if $ac=ca,$ does it follow that $a^*c=ca^*?$ More generally, if $a$ and $b$ are normal and if $ac=cb,$ does it follow that $a^*c=cb^*?$
\end{prob}

If the operators $a$ and $c$ belong to a finite factor $\mathcal{M},$ then the first part of the problem was resolved (in the affirmative) by von Neumann himself. In full generality, a problem was resolved by Fuglede \cite{F}.

Furthermore, von Neumann mentioned that a \lq\lq formal\rq\rq \ analogue of Problem \ref{von neumann problem} for unbounded operators can be {\it non-rigorously} answered in the negative due to the fact that a product of $2$ unbounded operators does not always exists. A partial affirmative answer was given by Putnam (see Theorem 1.6.2 in \cite{P}). He proved that if $cb\subset ac,$ then $cb^*\subset a^*c$ provided that $c$ is {\it bounded}.

In what follows, we propose a rigorous analogue of the Problem \ref{von neumann problem} for unbounded operators affiliated with a von Neumann algebra $\mathcal{M}.$ We start with a proper framework.

The set of all operators affiliated to a von Neumann algebra $\mathcal{M}$ does not necessarily form an algebra. At the same time, the class of unital $\ast$-algebras\footnote{The operations in these algebras are strong sum, strong product, the scalar multiplication and the usual adjoint of operators. For precise definitions, see Section 2.} which consist of operators affiliated with $\mathcal{M}$ is vast. In particular, it contains all algebras of measurable operators \cite{Seg}, and those of $\tau$-measurable operators \cite{Ne}.

According to \cite{CZ}, in this class, there is a unique maximal element called $LS(\mathcal{M}).$ We call $LS(\mathcal{M})$ the algebra of all locally measurable operators affiliated with $\mathcal{M}.$ An equivalent constructive definition of $LS(\mathcal{M})$ is given in Section \ref{prelims}.

We now properly restate Problem \ref{von neumann problem} for unbounded operators affiliated with $\mathcal{M}.$

\begin{prob}\label{our problem} Let $\mathcal{M}$ be a von Neumann algebra and let $a,b,c\in LS(\mathcal{M}).$ If $a$ and $b$ are normal and if $ac=cb,$ does it follow that $a^*c=cb^*?$
\end{prob}

Theorem 5 in \cite{B} delivers the positive answer to Problem \ref{our problem} for the case when $a,$ $b$ and $c$ are measurable operators affiliated with a von Neumann algebra $\mathcal{M}$ of type I (see also \cite{ACM}). In the case of an arbitrary finite von Neumann algebra $\mathcal{M},$ the Problem \ref{our problem} is resolved in the affirmative in \cite{HSWY} (see Corollary 3.6 there).

We answer Problem \ref{our problem} in the affirmative in full generality. Our methods are stronger than those of \cite{ACM}, \cite{B}, \cite{F}, \cite{HSWY}, \cite{P} and are of independent interest.

The following theorem is the main result of the paper.

\begin{thm}\label{t3}
Let $\mathcal{M}$ be an arbitrary von Neumann algebra, let  $a,b,c$ be locally measurable operators affiliated with $\mathcal{M}.$ If $a$ and $b$ are normal and if $ac=cb,$ then $a^*c=cb^*.$
\end{thm}

The corollary below extends the classical spectral theorem for normal operator (see e.g. \cite[Ch. 13, Theorem 13.33]{R}) to the setting of locally measurable operators.

\begin{cor}\label{t5}
Let $\mathcal{M}$ be an arbitrary von Neumann algebra and let $a,b$ be locally measurable operators affiliated with $\mathcal{M}.$ If $a$ is normal and $ab=ba$ then  $eb=be$ for every spectral projection $e$ of the operator $a$. If $a$ and $b$ are normal, then the following conditions are equivalent:
\begin{enumerate}[{\rm (a)}]
\item\label{cora} $ab=ba$;
\item\label{corb} $ef=fe$ for every spectral projection $e$ of the operator $a$ and for every spectral projection $f$ of the operator $b;$
\item\label{corc} $\phi(a)\psi(b)=\psi(b)\phi(a)$ for every Borel complex  functions $\phi$ and $\psi$ on $\mathbb C$, which bounded on compact subsets.
\end{enumerate}
\end{cor}

\section{Preliminaries}\label{prelims}

Let $H$ be a Hilbert space, let $B(H)$ be the $*$-algebra of all
bounded linear operators on $H$, and let $\mathbf{1}$ be the
identity operator on $H$. Given a von Neumann algebra
$\mathcal{M}$ acting on $H$, denote by $\mathcal{Z}(\mathcal{M})$
the centre of $\mathcal{M}$ and by
$\mathcal{P}(\mathcal{M})=\{p\in\mathcal{M}:\ p=p^2=p^*\}$ the
lattice of all projections in $\mathcal{M}$. Let
$\mathcal{P}_{fin}(\mathcal{M})$ be the set of all finite projections in
$\mathcal{M}$.

A linear operator $a:\mathfrak{D}\left( a\right) \rightarrow
H $, where the domain $\mathfrak{D}\left( a\right) $ of $a$ is a linear
subspace of $H$, is said to be {\it affiliated} with $\mathcal{M}$ if $ba \subseteq
ab$ for all $b$ from the commutant $\mathcal{M}^{\prime }$ of algebra $\mathcal{M}$.

A densely-defined closed linear operator $a$ (possibly unbounded)
affiliated with $\mathcal{M}$  is said to be \emph{measurable}
with respect to $\mathcal{M}$ if there exists a sequence
$\{p_n\}_{n=1}^\infty\subset \mathcal{P}(\mathcal{M})$ such that
$p_n\uparrow \mathbf{1},\ p_n(H)\subset \mathfrak{D}(a)$ and
$p_n^\bot=\mathbf{1}-p_n\in \mathcal{P}_{fin}(\mathcal{M})$ for every
$n\in\mathbb{N}$, where
$\mathbb{N}$ is the set of all natural numbers. Let us denote by
$S(\mathcal{M})$ the set of all measurable operators.

Let $a,b\in S(\mathcal{M})$. It is well known that $a+b,\ ab$ and $a^*$ are
densely-defined and preclosed operators. Moreover, the closures
$\overline{a+b}$ (strong sum), $\overline{ab}$ (strong product)
and $a^{\ast}$ are also measurable, and equipped with this operations $S(\mathcal{M})$ is a unital
$\ast$-algebra over the field $\mathbb{C}$ of complex numbers  \cite{Seg}. It is clear that $\mathcal{M}$ is a
$\ast$-subalgebra of $S(\mathcal{M})$.

A densely-defined linear operator $a$ affiliated with
$\mathcal{M}$ is called \emph{locally measurable} with respect to
$\mathcal{M}$ if there is a sequence $\{z_n\}_{n=1}^\infty$ of
central projections in $\mathcal{M}$ such that $z_n \uparrow
\mathbf{1},\ z_n(H)\subset\mathfrak{D}(a)$ and $az_n\in
S(\mathcal{M})$ for all $n\in\mathbb{N}$.

The set $LS(\mathcal{M})$ of all locally measurable operators
 is a unital $\ast$-algebra over the
field $\mathbb{C}$ with respect to the same algebraic operations
as in $S(\mathcal{M})$ \cite{Yead} and $S(\mathcal{M})$ is a
$\ast$-subalgebra of $LS(\mathcal{M})$. It is clear that if $\mathcal{M}$ is finite, the algebras
$S(\mathcal{M})$ and $LS(\mathcal{M})$ coincide. If von Neumann algebra $\mathcal{M}$
is of type $III$ and $\dim(\mathcal{Z}(\mathcal{M}))=\infty$, then
$S(\mathcal{M})=\mathcal{M}$ but $LS(\mathcal{M})\neq
\mathcal{M}$.

For every subset $E\subset LS(\mathcal{M})$, the sets of all
self-adjoint (resp., positive) operators in $E$ will be denoted by
$E_h$ (resp. $E_+$). The partial order in $LS(\mathcal{M})$ is
defined by its cone $LS_+(\mathcal{M})$ and is denoted by $\leq$.

Let $a$ be a closed operator with dense domain $\mathfrak{D}(a)$
in $H$, let $a=u|a|$ be the polar decomposition of the operator $a$,
where $|a|=(a^*a)^{\frac{1}{2}}$ and $u$ is a  partial isometry
in $B(H)$ such that $u^*u$ (respectively, $uu^*$) is the right (left) support $r(a)$  (respectively, $l(a)$) of $a$. It is
known that $a = |a^*|u$ and  $a\in LS(\mathcal{M})$ (respectively, $a\in S(\mathcal{M})$) if and only if $|a|\in
LS(\mathcal{M})$ (respectively, $|a|\in
S(\mathcal{M})$) and $u\in \mathcal{M}$~\cite[\S\S\,2.2,2.3]{MCh}. If
$a$ is a self-adjoint operator affiliated with $\mathcal{M}$, then
the spectral family of projections $e_\lambda(a)= e_{(-\infty,\lambda]}(a), \ {\lambda\in
  \mathbb{R}}$, for $a$ belongs to
$\mathcal{M}$~\cite[\S\,2.1]{MCh}. A locally measurable operator $a$
is measurable if and only if
$e_\lambda^\bot(|a|)\in \mathcal{P}_{fin}(\mathcal{M})$ for some $\lambda>0$ \cite[\S\,2.2]{MCh}.

In what follows, we use the notation  $n(a)=\mathbf{1}-r(a)$ for the projection onto the kernel of the operator $a$.

Assume now that $\mathcal{M}$ is a semifinite von Neumann algebra equipped with a faithful
normal semifinite trace $\tau$. A densely-defined closed linear operator $a$ affiliated with $\mathcal{M}$ is called  $\tau$-measurable if for each $\varepsilon >0$
there exists  $e\in \mathcal{P}(\mathcal{M})$ with $\tau (e^{\perp})\leq \varepsilon$ such that $e(H)\subset \mathfrak{D}(a)$.
Let us denote by $S(\mathcal{M},\tau)$ the set of all $\tau$-measurable operators. It is well-known \cite{Ne} that $S(\mathcal{M},\tau)$ is a $\ast$-subalgebra of $S(\mathcal{M})$ and $\mathcal{M} \subset S(\mathcal{M},\tau)$. It is clear that if $\mathcal{M}$ is a semifinite factor, the algebras $S(\mathcal{M},\tau)$ \ and  $S(\mathcal{M})$  coincide. Note also that for every  $a \in S(\mathcal{M},\tau)$ there exists $\lambda>0$  such that
$\tau(e_\lambda^\bot(|a|)) < \infty$ (see \cite{Ne} and \cite[\S\,2.6]{MCh}).

Measure topology is defined in $S(\mathcal{M},\tau)$ by the family $V(\varepsilon,\delta)$ $\varepsilon>0,\delta>0,$ of neighborhoods of zero
$$V(\varepsilon,\delta)=\{ a\in S(\mathcal{M},\tau) : \|ae\|_{\mathcal{M}}\leq \delta \text { \ for some } e\in \mathcal{P}(\mathcal{M}) \text {\ with } \tau (e^{\perp})\leq \varepsilon \}$$
Convergence of the sequence $\{a_n\} \subset S(\mathcal{M},\tau)$ in measure topology is called {\it convergence in measure}. When equipped with measure topology, $S(\mathcal{M},\tau)$ is a complete metrizable topological $\ast$-algebra (see \cite{Ne}). For basic properties of the measure topology, see \cite{Ne}. We remark only that $e_n\rightarrow 0$ in measure, $e_n\in \mathcal{P}(\mathcal{M}),$ if and only if $\tau(e_n)\rightarrow 0.$

Let $\mathcal{M}$ be a von Neumann algebra equipped with a faithful normal semifinite trace $\tau.$ We set
$$(L_1\cap L_{\infty})(\mathcal{M},\tau)=\Big\{x\in\mathcal{M}:\ \tau(|x|)<\infty\Big\}.$$



The following property is standard.

\begin{property}\label{t1} Let $\mathcal{M}$ be a semifinite von Neumann algebra and let $\tau$ be a faithful normal semifinite trace on $\mathcal{M}.$ If $x,y\in (L_1\cap L_{\infty})(\mathcal{M},\tau),$ then $\tau(xy)=\tau(yx).$
\end{property}


\section{The Fuglede-Putnam theorem in  $S(\mathcal{M},\tau)$}\label{tau measurable section}

The proof of Theorem \ref{t3} in full generality is based on its special case for the $\ast$-algebra $S(\mathcal{M},\tau).$


\begin{thm} \label{t2} Let $\mathcal{M}$ be a semifinite von Neumann algebra and let $\tau$ be a faithful normal semifinite trace on $\mathcal{M}$. Let $a,b,c\in S(\mathcal{M},\tau),$ and let $a$ and $b$ be normal. If $ac=cb,$ then $a^*c=cb^*.$
\end{thm}

Our strategy of the proving Theorem \ref{t2} relies on a number of auxiliary lemmas. Whereas some of them look similar to those used in \cite{HSWY}, the lemmas below appear to be stronger than their counterparts from \cite{HSWY}.

\begin{lem}\label{l1} If $a\in\mathcal{M}$ is normal,  $p \in \mathcal{P}(\mathcal{M})$ and $\tau(p) < \infty$  with $ap=pap$, then $ap=pa.$
\end{lem}
\begin{proof} Denote, for brevity,
$$a_1=pap,\quad a_2=pa(\mathbf{1}-p).$$
Due to the normality of $a$ and using the equality $ap=pap$, we have
$$a_1^*a_1=(ap)^*(ap)=pa^*ap=paa^*p=(pa)(pa)^*=(a_1+a_2)(a_1+a_2)^*=a_1a_1^*+a_2a_2^*.$$
Since $\tau(p) < \infty$, it follows that $a_1,a_2\in(L_1\cap L_{\infty})(\mathcal{M},\tau).$ Taking the trace and using Property \ref{t1}, we conclude that $\tau(a_2a_2^*)=0.$ Since $\tau$ is faithful, it follows that $a_2=0.$ This completes the proof.
\end{proof}

\begin{lem}\label{l2} Let $a,b\in\mathcal{M}$ are normal, $c\in S_+(\mathcal{M},\tau)$ and $ac=cb$. Let $\lambda>0$ be such that $\tau(e_{(\lambda,+\infty)}(c)) < \infty$. Setting $p_1=e_{[0,\lambda]}(c)$ and $p_2=e_{(\lambda,+\infty)}(c),$ we obtain
$$p_ia=ap_i,\ p_ib=bp_i,\ i=1,2.$$
\end{lem}
\begin{proof} Set $a_{ij}=p_iap_j$ and $b_{ij}=p_ibp_j$ for $i,j=1,2.$

{\bf Step 1:} We claim that
$$\tau(a_{12}^*a_{12})=\tau(a_{21}^*a_{21}),\quad \tau(b_{12}^*b_{12})=\tau(b_{21}^*b_{21}).$$

Indeed, use the equality $a^*a = aa^*$ and Property \ref{t1}, we have
$$\tau(a_{12}^*a_{12})=\tau(p_2a^*p_1ap_2)=\tau(p_2a^*ap_2)-\tau(p_2a^*p_2ap_2)=$$
$$=\tau(p_2aa^*p_2)-\tau(p_2ap_2a^*p_2)=\tau(p_2ap_1a^*p_2)=\tau(a_{21}a_{21}^*)=\tau(a_{21}^*a_{21}).$$
The proof of the second equality in the claim is identical.

{\bf Step 2:} We claim that
$$\tau(a_{12}^*a_{12})\leq\tau(b_{12}^*b_{12})$$
and
$$\tau(b_{21}^*b_{21})\leq\tau(a_{21}^*a_{21}).$$
Since $p_2c^2p_2\geq \lambda^2p_2,$ it follows that
$$a_{12}a_{12}^*=p_1a\cdot p_2\cdot a^*p_1\leq \lambda^{-2}\cdot p_1a\cdot p_2c^2p_2\cdot a^*p_1=\lambda^{-2}(a_{12}c)(a_{12}c)^*.$$
By assumption,
\begin{equation}\label{ch1}
a_{12}c=p_1ap_2\cdot c=p_1\cdot ac\cdot p_2=p_1\cdot cb\cdot p_2=c\cdot p_1bp_2=cb_{12}.
\end{equation}
Thus,
\begin{equation}\label{ch6}
a_{12}a_{12}^*\leq\lambda^{-2}(cb_{12})(cb_{12})^*.
\end{equation}
Since $cp_1\in\mathcal{M}$ and since $b_{12}\in(L_1\cap L_{\infty})(\mathcal{M},\tau),$ it follows that
$$cb_{12}=cp_1\cdot b_{12}\in (L_1\cap L_{\infty})(\mathcal M,\tau).$$
Hence (see Property \ref{t1}),
$$\tau(a_{12}^*a_{12})=\tau(a_{12}a_{12}^*)\stackrel{\eqref{ch6}}{\leq}\lambda^{-2}\tau((cb_{12})(cb_{12})^*)=
\lambda^{-2}\tau((cb_{12})^*(cb_{12}))<\infty.$$
Using now the inequality  $p_1c^2p_1\leq \lambda^2p_1$, we have that
$$(cb_{12})^*(cb_{12})=b_{12}^*c^2b_{12}=b_{12}^*\cdot p_1c^2p_1\cdot b_{12}\leq\lambda^2\cdot b_{12}^*p_1b_{12}=\lambda^2b_{12}^*b_{12}.$$
and
\begin{equation}\label{ch2}
\tau(a_{12}^*a_{12}) \leq \tau(b_{12}^*b_{12}).
\end{equation}
Let $a'=b^*$ and $b'=a^*.$ Taking the adjoints in the equality $ac=cb,$ we obtain $a'c=cb'.$ In addition
$$a_{12}'\stackrel{def}{=}p_1a'p_2=b_{21}^*,\quad b_{12}'\stackrel{def}{=}p_1b'p_2=a_{21}^*.$$
Applying \eqref{ch2} to the triple $(a',b',c),$ we obtain
$$\tau(b_{21}^*b_{21})=\tau(b_{21}b_{21}^*) =\tau((a'_{12})^{\ast}a'_{12})\leq \tau((b'_{12})^{\ast}b'_{12}) =\tau(a_{21}a_{21}^*)=\tau(a_{21}^*a_{21}).$$
This proves the claim.

{\bf Step 3:}
Using Steps 1, 2, we obtain
$$\tau(a_{12}^*a_{12})\leq\tau(b_{12}^*b_{12})=\tau(b_{21}^*b_{21})\leq\tau(a_{21}^*a_{21})=\tau(a_{12}^*a_{12}).$$
Thus,
$$\tau(a_{12}^*a_{12})=\tau(a_{21}^*a_{21})=\tau(b_{12}^*b_{12})=\tau(b_{21}^*b_{21}).$$

{\bf Step 4:} We claim that $ap_2=p_2a$ and $bp_2=p_2b$.

By \eqref{ch1}, we have
$$a_{12}c=cb_{12}.$$
Now, using Property \ref{t1}, we obtain
$$\tau((a_{12}c)(a_{12}c)^*)=\tau((cb_{12})(cb_{12})^*)=\tau((cb_{12})^*(cb_{12})).$$
Definition of $p_1$ now yields
$$(cb_{12})^*(cb_{12})=b_{12}^*c^2b_{12}=b_{12}^*\cdot p_1c^2p_1\cdot b_{12}\leq\lambda^2 b_{12}^*\cdot p_1\cdot b_{12}=\lambda^2b_{12}^*b_{12}.$$
It follows from Step 3 that
$$\tau((a_{12}c)(a_{12}c)^*)\leq \lambda^2\tau(a_{12}^*a_{12}).$$
In other words,
$$\tau(a_{12}\cdot p_2c^2p_2\cdot a_{12}^*)=\tau((a_{12}c)(a_{12}c)^*)\leq \lambda^2\tau(a_{12}a_{12}^*)=\tau(a_{12}\cdot \lambda^2p_2\cdot a_{12}^*).$$
Hence,
$$\tau(a_{12}\cdot p_2(c^2-\lambda^2 \mathbf{1})p_2\cdot a_{12}^*)\leq0.$$
Since $\tau$ is faithful and since
$$a_{12}\cdot p_2(c^2-\lambda^2 \mathbf{1})p_2\cdot a_{12}^*\geq 0,$$
it follows that
\begin{align}\label{ch3}
a_{12}\cdot p_2(c^2-\lambda^2 \mathbf{1})p_2\cdot a_{12}^*=0.
\end{align}

For every $\varepsilon >0$ and $p_\varepsilon =e_{(\lambda+\varepsilon,+\infty)}(c)$ we have
$c^2 p_\varepsilon = p_\varepsilon c^2 p_\varepsilon \geq (\lambda + \varepsilon)^2 p_\varepsilon.$ Therefore,
$$c^2 p_2\geq \lambda^2p_2+\varepsilon^2p_{\varepsilon},\quad p_2(c^2-\lambda^2 \mathbf{1})p_2\geq \epsilon^2p_\varepsilon.$$
We now infer from \eqref{ch3} that
$$a_{12}\cdot p_\varepsilon\cdot a_{12}^*=0.$$
Since $p_\varepsilon \rightarrow p_2$ in measure as $\epsilon\to 0,$ we obtain $a_{12}a_{12}^*=0.$ Thus, $a_{12}=0$ and
$$ap_2 = p_1ap_2 + (\mathbf{1}-p_1)ap_2 = (\mathbf{1}-p_1)ap_2 = p_2ap_2.$$
Hence, we infer from Lemma \ref{l1} that $ap_2=p_2a.$ Similarly, $bp_2=p_2b.$
\end{proof}

\begin{lem}\label{l3} Let $a,b\in\mathcal{M}$ be normal, $c\in S_+(\mathcal{M},\tau)$ and $ac=cb$. Then $a^*c=cb^*.$
\end{lem}
\begin{proof}
The assumption  $c\in S_+(\mathcal{M},\tau)$ guarantees that there exists $\lambda>0$ such that  $\tau(\mathbf{1}-e_\lambda(c))<\infty$.
Set $p_2=e_{(\lambda,+\infty)}(c), \ p_1 = (\mathbf{1} - p_2)$ and $a_j=ap_j,$ $b_j=bp_j,$ $c_j=cp_j,$ $j=1,2$.
By Lemma \ref{l2},  the operators $a$ and $b$ commute with $p_j$, in particular, $a_j$ and $b_j$ are normal $j=1,2$.
By the same lemma, the operator $a$ commutes with projections  $(\mathbf{1}-e_\nu(c))$ for all $\nu \geq \lambda$.
Since finite linear combinations of projections $(\mathbf{1}-e_\nu(c)),$ $\nu \geq \lambda$, converge to operator
$c_2$ in the measure topology and since multiplication in $S(\mathcal{M},\tau)$ is continuous in that topology,  it follows that
\begin{align}\label{ch5}
c_2a=ac_2\mbox{ and, similarly, }c_2b=bc_2.
\end{align}
Appealing now to Lemma \ref{l2}, we obtain
\begin{align}\label{ch4}
c_2a=ac_2=ac\cdot p_2=cb\cdot p_2=c\cdot bp_2\stackrel{L.\ref{l2}}{=}c\cdot p_2b=cp_2\cdot b=c_2b.
\end{align}
Combining now (\ref{ch4}) and (\ref{ch5}) yields
$$
a^*c_2=(c_2a)^*=(c_2b)^*= (bc_2)^*= c_2b^*.
$$

Taking \eqref{ch5} into account, we rewrite \eqref{ch4} as $ac_2=c_2b.$ Combining this with the assumption $ac=cb,$ we infer $a c_1 = c_1 b.$ Taking into account that
$c_1\in\mathcal{M}$ and applying the classical Fuglede-Putnam theorem we derive that
$$a^*c_1=c_1b^*.$$
Thus,
$$a^*c=a^*c_1+a^*c_2= c_1b^*+c_2b^*=cb^*.$$
\end{proof}

\begin{lem}\label{ber last lemma} Let $a,b\in\mathcal{M}$ be normal and let $c\in S(\mathcal{M},\tau)$ be such that $ac=cb.$ If $n(c^*)\preceq n(c)$ or $n(c)\preceq n(c^*),$ then $a^*c=cb^*.$
\end{lem}
\begin{proof} We only consider the first case (the second case can be reduced to the first one by considering the triple $(b^*,a^*,c^*)$ instead of the triple $(a,b,c)$).

Let $c=v|c|$ be a polar decomposition of $c$ so that $v^*v=r(c)$ and $vv^*=r(c^*)$. Let $w$ be a partial isometry such that $w^*w=n(c^*)$ and $ww^*\leq n(c).$ Define an isometry $u=v^*+w$ (that is, $u^*u=1$). It is immediate that $u^*|c|=c$ and $uc=|c|.$ Thus,
$$(uau^*)\cdot |c|=ua\cdot c=u\cdot ac=u\cdot cb=uc\cdot b=|c|\cdot b.$$
Since $u^*u=1$ and since $a$ is normal, it follows that $uau^*$ is also normal. Applying Lemma \ref{l3} to the triple $(uau^*,b,|c|),$ we obtain
$$(ua^*u^*)\cdot |c|=|c|\cdot b^*.$$
Therefore,
$$a^*c=a^* \cdot u^*|c|=u^*\cdot (ua^*u^*)\cdot |c|=u^*\cdot |c|\cdot b^*=cb^*.$$
This completes the proof.
\end{proof}

We give now the proof of Theorem \ref{t2} in the  case of arbitrary semifinite von Neumann algebra $\mathcal{M}$ with a faithful normal semifinite trace $\tau$.

\begin{proof}[Proof of Theorem \ref{t2}] Let us suppose at first that $a,b\in\mathcal{M}.$ By \cite[Theorem 2.1.3]{Sak} there exist central projectors $z_1,z_2\in \mathcal{Z}(\mathcal{M})$ such that
$$z_1+z_2=\mathbf{1},\quad n(c^*)z_1\preceq n(c)z_1,\ n(c)z_2\preceq n(c^*)z_2.$$
It is immediate that
$$az_1\cdot cz_1=cz_1\cdot bz_1,\quad az_2\cdot cz_2=cz_2\cdot bz_2.$$
Clearly, $n(cz_k)=n(c)z_k$ and $n(c^*z_k)=n(c^*)z_k,$ $k=1,2,$ where the left hand side is taken in the algebra $z_k\mathcal{M}.$ Applying Lemma \ref{ber last lemma} to the triples $(az_1,bz_1,cz_1)$ and $(az_2,bz_2,cz_2),$ we obtain
$$a^*z_1\cdot cz_1=cz_1\cdot b^*z_1,\quad a^*z_2\cdot cz_2=cz_2\cdot b^*z_2.$$
Summing these equalities, we obtain that $a^*c=cb^*.$ This proves the assertion for the case $a,b\in\mathcal{M}.$

Let now  $a,b$ be an arbitrary normal operators in $S(\mathcal{M},\tau)$ and $ac = cb$. Let $q_n$ (respectively, $r_n$) be the spectral projection for $a$ (respectively, $b$) corresponding to the set $\{z: |z|\leq n\}$. It is clear that $\{q_n\}$ and $\{r_n\}$  are an increasing sequences of projections with $\sup_{n\geq 1}q_n=\mathbf{1}$ and $\sup_{n\geq1}r_n=\mathbf{1}.$ In addition (see e.g. \cite[Ch. 13, Theorems 13.24, 13.33]{R}),
$$aq_n=q_na,\quad a^*q_n=q_na^*,\quad br_n=r_nb,\quad b^*r_n=r_nb^*,\quad n\in\mathbb{N}.$$
Multiplying the equality $ac=cb$ by $q_n$ on the left and by $r_n$ on the right, we obtain
$$(q_na)\cdot (q_ncr_n)=(q_ncr_n)\cdot (r_nb),\quad n\in\mathbb{N}.$$
Clearly, $q_na\in\mathcal{M}$ and $r_nb\in\mathcal{M}$ are normal operators for every $n\in\mathbb{N}.$ It follows from the preceding paragraph that
$$q_n\cdot a^*c\cdot r_n=(q_na)^*\cdot (q_ncr_n)=(q_ncr_n)\cdot (r_nb)^*=q_n\cdot cb^*\cdot r_n .$$
Thus,
$$q_n(a^*c-cb^*)r_n=0,\quad n\in\mathbb{N}.$$
Since  $a,b\in S(\mathcal{M},\tau),$ it follows that $\tau(1-q_n)\rightarrow 0,$ $\tau(1-r_n)\rightarrow 0$ as $n\to\infty.$ Thus  $q_n \rightarrow \mathbf{1},$ $r_n \rightarrow \mathbf{1}$ in measure. Therefore, for every $x\in S(\mathcal{M},\tau),$ we have $q_nxr_n\to x$ in measure as $n\to\infty.$ Taking $x=a^*c-cb^*,$ we complete the proof.
\end{proof}

\section{The Fuglede-Putnam theorem in the $\ast$-algebra $LS(\mathcal{M})$}\label{locally measurable section}

Lemma \ref{l5} below is the key tool used to extend Fuglede-Putnam theorem from $\tau-$measurable operators to measurable ones.

\begin{lem}\label{l5}
Let $\mathcal{M}$ be a semifinite von Neumann algebra and let $q \in \mathcal{P}(\mathcal{M})$ be a finite projector. Then there exists partition of unity  $\{z_i\}_{i\in I} \subset \mathcal{P}(\mathcal{Z}(\mathcal{M}))$, such that  every von Neumann algebra  $z_i\mathcal{M}, \ i \in I$, has  a faithful normal semifinite trace $\tau_i$ with $\tau_i(z_iq)<\infty$.
\end{lem}
\begin{proof}
It is well known that let $\mathcal{M}$ a commutative von Neumann algebra $\mathcal{Z}(\mathcal{M})$ is $\ast$-isomorphic to the $\ast$-algebra
$L^\infty(\Omega,\Sigma,\mu)$ of all essentially bounded
measurable complex-valued functions defined on a measure space
$(\Omega,\Sigma,\mu)$ with the measure $\mu$ satisfying the direct
sum property (we identify functions that are equal almost
everywhere) (see e.g. \cite[Ch. 7, \S 7.3]{Dixmier}).  The direct sum property of a measure $\mu$ means that the Boolean algebra of all projections of the $*$-algebra
$L^\infty(\Omega,\Sigma,\mu)$ is order complete, and for any
non-zero $ p\in \mathcal{P}(\mathcal{M})$ there exists a non-zero
projection $r\leq p$ such that $\mu(r)<\infty$.  The direct sum property
of a measure $\mu$ is equivalent to the fact that the functional
$\nu(f):=\int_\Omega f\,d\mu$ is a semi-finite normal faithful trace on the algebra $L^\infty(\Omega,\Sigma,\mu)$. Therefore there exists partition of unity  $\{r_j\}_{j\in j} \subset \mathcal{P}(L^\infty(\Omega,\Sigma,\mu))$, such that  $\nu_j(f) = \nu(r_j f) $ is faithful normal finite trace on $r_j L^\infty(\Omega,\Sigma,\mu)$ for every $j \in J$.

Let $\varphi$ be a $\ast$-isomorphism from $\mathcal{Z}(\mathcal{M})$
onto the $*$-algebra $L^\infty(\Omega,\Sigma,\mu)$.  Denote by
$L^+(\Omega,\, \Sigma,\, m)$ the set of all measurable real-valued
functions defined on $(\Omega,\Sigma,\mu)$ and taking values in
the extended half-line $[0,\, \infty]$ (functions that are equal
almost everywhere are identified).

By \cite[Ch. V, \S 2, Theor. 2.34 and Prop. 2.35]{Tak} there exits a faithful semifinite normal extended center valued trace $T$
$$
T\colon
\mathcal{M}_+\to L^+(\Omega,\Sigma,\mu)
$$
such that $\mu(\{\omega \in \Omega:T(q)(\omega)=+\infty\})=0$.
Thus characteristic functions $q_n=\chi_{A_n}$ corresponding to sets $A_n=\{\omega \in \Omega:n-1\leq T(q)(\omega)< n\},$ $n\in\mathbb{N},$  is partition of the unit element $\chi_{\Omega}$  of Boolean algebra $\mathcal{P}(L^\infty(\Omega,\Sigma,\mu))$. In addition
$$
T(q\varphi^{-1}(q_n)) = \varphi^{-1}(q_n) T(q) \leq nq_n
$$
for all $n \in \mathbb N$.

It is clear that $\{z_n^j=\varphi^{-1}(r_j q_n), \ j \in J, \ n \in \mathbb N\}$ is partition of unity  in  $\mathcal{P}(\mathcal{Z}(\mathcal{M}))$. In addition, the functional $\tau_{j,n}:z_n^j\mathcal{M}_+\to[0,\infty],$ given by the formula
$$\tau_{j,n}(x) = \nu_j(T(x)), \ x \in  z_n^j\mathcal{M}_+$$
is a  faithful normal finite trace on $z_n^j\mathcal{M}$.
In particular,
$$\tau_{j,n}(z_n^jq)=\nu_j(T(z_n^jq))\leq n\nu_j(\varphi^{-1}(q_n)r_j)\leq n\nu_j(r_j)<\infty\mbox{ for all }j \in J,\quad n\in \mathbb{N}.$$
Setting $i=(j,n)$ and $I=J\times\mathbb{N},$ we complete the proof.
\end{proof}

\begin{lem}\label{central partition} Let $\mathcal{M}$ be a von Neumann algebra and let $\{z_i\}_{i\in I}\subset\mathcal{Z}(\mathcal{M})$ be a partition of unity. If $x\in LS(\mathcal{M})$ is such that $xz_i=0$ for every $i\in I,$ then $x=0.$
\end{lem}
\begin{proof} Since $z_i \leq n(x)$  for all   $i\in I,$ it follows that $\mathbf{1}=\sup_{i\in I}z_i \leq n(x).$ Thus $n(x)=\mathbf{1},$ i.e. $x=0.$
\end{proof}

The following lemma extends the result of Theorem \ref{t2} to the setting of measurable operators.

\begin{lem}\label{semifinite measurable} Let $\mathcal{M}$ be semifinite von Neumann algebra and let $a,b,c\in S(\mathcal{M}).$ If $a$ and $b$ are normal and if $ac=cb,$ then $a^*c=cb^*.$
\end{lem}
\begin{proof} Choose $n$ so large that projections $e_{|a|}(n,+\infty),$ $e_{|b|}(n,+\infty)$ and $e_{|c|}(n,+\infty)$ are finite. Let $q$ be a finite projection given by the formula
$$q=e_{|a|}(n,+\infty)\vee e_{|b|}(n,+\infty)\vee e_{|c|}(n,+\infty).$$
Let $\{z_i\}_{i\in I}$ be the partition of unity constructed in Lemma \ref{l5}. We have
$$az_i\cdot cz_i=cz_i\cdot bz_i,\quad i\in I.$$
It follows from Lemma \ref{l5} that, for a given $i\in I,$
$$\tau_i(e_{|a|}(n,+\infty)z_i),\tau_i(e_{|b|}(n,+\infty)z_i),\tau_i(e_{|c|}(n,+\infty)z_i)<\infty.$$
A standard argument yields
$$e_{|a|}(n,+\infty)z_i=e_{|az_i|}(n,+\infty),$$
where the right hand side is taken in the algebra $z_i\mathcal{M}.$ It follows that $az_i,$ $bz_i$ and $cz_i$ are $\tau_i-$measurable operators for every $i\in I.$ Theorem \ref{t2} implies that
$$a^*z_i\cdot cz_i=cz_i\cdot b^*z_i.$$
The assertion follows now from Lemma \ref{central partition}.
\end{proof}

Lemma \ref{semifinite locally measurable} extends Fuglede-Putnam theorem to the setting of locally measurable operators affiliated with a semifinite von Neumann algebra $\mathcal{M}.$

\begin{lem}\label{semifinite locally measurable} Let $\mathcal{M}$ be semifinite von Neumann algebra and let $a,b,c\in LS(\mathcal{M}).$ If $a$ and $b$ are normal and if $ac=cb,$ then $a^*c=cb^*.$
\end{lem}
\begin{proof} By the (constructive) definition of the algebra $LS(\mathcal{M}),$ there exist central projections $\{p_k\}_{k\geq1},$ $\{q_l\}_{l\geq1}$ and $\{r_m\}_{m\geq1}$ such that $p_k\uparrow \mathbf{1},$ $q_l\uparrow \mathbf{1}$ and $r_m\uparrow \mathbf{1}$ and such that
$$ap_k,bq_l,cr_m\in S(\mathcal{M}),\quad k,l,m\geq1.$$
Denote the triple $(k,l,m)$ by $n$ and set $P_n=p_kq_lr_m.$ Since
$$aP_n\cdot cP_n=cP_n\cdot bP_n,\quad n\in\mathbb{N}^3,$$
it follows from Lemma \ref{semifinite measurable} that
$$a^*P_n\cdot cP_n=cP_n\cdot b^*P_n,\quad n\in\mathbb{N}^3.$$
In other words (here, we let $r_0=0$),
$$(a^*c-cb^*)p_kq_l\cdot (r_m-r_{m-1})=0,\quad m\in\mathbb{N}.$$
Since $\{r_m-r_{m-1}\}_{m\geq1}$ is a partition of unity which consists of central projections, it follows from Lemma \ref{central partition} that
$$(a^*c-cb^*)p_kq_l=0,\quad k,l\in\mathbb{N}.$$
Repeating the argument for $l$ and, after that, for $k,$ we complete the proof.
\end{proof}

The following assertion can be found in \cite{ACM} (see Theorem 1 there). We provide a short proof for convenience of the reader.

\begin{lem}\label{purely infinite} Let $\mathcal{M}$ be a purely infinite von Neumann algebra and let $a,b,c\in LS(\mathcal{M}).$ If $a$ and $b$ are normal and if $ac=cb,$ then $a^*c=cb^*.$
\end{lem}
\begin{proof} Recall that $S(\mathcal{M})=\mathcal{M}.$ Choose central projections $\{p_k\}_{k\geq1},$ $\{q_l\}_{l\geq1}$ and $\{r_m\}_{m\geq1}$ such that $p_k\uparrow \mathbf{1},$ $q_l\uparrow \mathbf{1}$ and $r_m\uparrow \mathbf{1}$ and such that
$$ap_k,bq_l,cr_m\in\mathcal{M},\quad k,l,m\geq1.$$
Denote the triple $(k,l,m)$ by $n$ and let $P_n=p_kq_lr_m.$ We have
$$aP_n\cdot cP_n=cP_n\cdot bP_n,\quad n\in\mathbb{N}^3.$$
By the classical Fuglede-Putnam theorem, we have
$$a^*P_n\cdot cP_n=cP_n\cdot b^*P_n,\quad n\in\mathbb{N}^3.$$
The same argument as in Lemma \ref{semifinite locally measurable} yields the assertion.
\end{proof}

\begin{proof}[Proof of Theorem \ref{t3}] It is well known that for every von Neumann algebra $\mathcal{M}$ there exist central projectors  $z_1,z_2\in\mathcal{Z}(\mathcal{M})$ such that $z_1+z_2=\mathbf{1},$ $\mathcal{M}z_1$ is the semifinite von Neumann algebra and $\mathcal{M}z_2$ is the purely infinite von Neumann algebra (see, for example, \cite[Ch.2, \S 2.2]{Sak}). We have
$$az_k\cdot cz_k=cz_k\cdot bz_k,\quad k=1,2.$$
Lemmas \ref{semifinite locally measurable} and \ref{purely infinite} imply that
$$a^*z_k\cdot cz_k=cz_k\cdot b^*z_k,\quad k=1,2.$$
Summing these equalities, we complete the proof.
\end{proof}

We need the following useful property of locally measurable operators.

\begin{lem}\label{l0} Let $\mathcal{M}$ be a von Neumann algebra and let $x\in LS(\mathcal{M}).$ Let $\{p_n\}_{n\geq1}\subset \mathcal P(\mathcal{M})$ be such that $p_n\uparrow \mathbf{1}$. If $p_nxp_n=0$ for every $n\geq1,$ then $x=0.$
\end{lem}
\begin{proof} Fix $m\in\mathbb{N}.$ For every $n\geq m,$ we have
$$p_mxp_n=p_m\cdot p_nxp_n=0.$$
Thus, $p_n\leq \mathbf{1} - r(p_mx)$ for every $n\geq1.$ Since $p_n\uparrow \mathbf{1},$ it follows that $r(p_mx)=0$ and, therefore, $p_mx=0.$

Hence, $x^*p_m=0$ for every $m\geq1.$ Thus, $p_m\leq \mathbf{1} - r(x^*)$ for every $m\geq1.$ Since $p_m\uparrow{\bf 1},$ it follows that $r(x^*)=0$ and, therefore, $x=0.$
\end{proof}

\begin{lem}\label{pervoe} Let $\mathcal{M}$ be a von Neumann algebra and let $a,b\in LS(\mathcal{M}).$ If $a$ is normal and if $ab=ba,$ then $eb=be$ for every spectral projection $e$ of the operator $a.$
\end{lem}
\begin{proof} Let $b_1=\Re(b)=\frac{b+b^*}{2}$ and $b_2=\Im(b)=\frac {b-b^*}{2i}.$ By Theorem \ref{t3} we have that $ab^*=b^*a.$ Thus $ab_j=b_ja,$ $j=1,2.$ Let a Borel function $\phi$ be given by the formula $\phi(t) = (t+i)^{-1},$ $t\in\mathbb{R},$ and let $c_j=\phi(b_j),$ $j=1,2.$ Since $b_j^*=b_j$ and since $|\phi(t)|\leq 1,$ $t\in\mathbb{R},$ it follows from Spectral Theorem that $c_j\in\mathcal{M},$ $j=1,2.$ Since $ab_j=b_ja,$ it follows that
$$a(b_j+i\mathbf{1})^{-1}-(b_j+i\mathbf{1})^{-1}a=(b_j+i\mathbf{1})^{-1}\cdot((b_j+i\mathbf{1})a-a(b_j+i\mathbf{1}))\cdot(b_j+i\mathbf{1})^{-1}=0,$$
that is, $ac_j=c_ja.$ Theorem 13.33 in \cite{R} yields that $ec_j=c_je,$ $j=1,2,$ for every spectral projection $e$ of the operator $a.$ Thus, $eb_j=(b_j+i\mathbf{1})c_je(b_j+i\mathbf{1})-ie=(b_j+i\mathbf{1})ec_j(b_j+i\mathbf{1})-ie=b_je,$ $j=1,2.$ Summing these equalities, we obtain $eb=be.$
\end{proof}

\begin{proof}[Proof of Corollary \ref{t5}] $\eqref{cora}\Rightarrow\eqref{corb}.$ Lemma \ref{pervoe} states that $eb=be$ for every spectral projection $e$ of the operator $a.$ Again applying Lemma \ref{pervoe} to the couple $(b,e),$ we obtain that $ef=fe$ for every spectral projection $e$ of the operator $a$ and for every spectral projection $f$ of the operator $b.$

$\eqref{corb}\Rightarrow\eqref{corc}.$ Let $q_n$ (respectively, $r_n$) be the spectral projection for $a$ (respectively, $b$) corresponding to the set $D_n=\{z: |z|\leq n\},$ $n\in\mathbb{N}.$ Denote $\phi_n=\phi\cdot\chi_{D_n}$ and $\psi_n=\psi\cdot\chi_{D_n}.$ By Spectral Theorem, we have
$$q_n\cdot\phi(a)=\phi(a)\cdot q_n=\phi_n(aq_n),\quad r_m\cdot\psi(b)=\psi(b)\cdot r_m=\psi_m(br_m).$$
Bounded operators $aq_n$ and $br_m$ are normal and their spectral projections commute. By the Spectral Theorem for bounded operators, these operators commute and, therefore,
$$\phi_n(aq_n)\cdot\psi_m(br_m)=\psi_m(br_m)\cdot\phi_n(aq_n).$$
Thus,
$$q_nr_m\cdot \phi(a)\psi(b)\cdot q_nr_m=q_nr_n\cdot \phi_n(aq_n)\psi_m(br_m)\cdot q_nr_m=$$
$$=q_nr_m\cdot\psi_m(br_m)\phi_n(aq_n)\cdot q_nr_m=q_nr_m\cdot\psi(b)\phi(a)\cdot q_nr_m.$$
Taking into account that $r_m\uparrow{\bf 1}$ and using Lemma \ref{l0}, we obtain
$$q_n\cdot \phi(a)\psi(b)\cdot q_n=q_n\cdot\psi(b)\phi(a)\cdot q_n.$$
Again appealing to Lemma \ref{l0}, we obtain \eqref{corc}.

Taking $\phi(z)=z$ and $\psi(z)=z$ in \eqref{corc}, we obtain the implication $\eqref{corc}\Rightarrow\eqref{cora}.$
\end{proof}


\begin{thebibliography}{99}

\bibitem{ACM} M. V. Ahramovich, V.I. Chilin  and M. A. Muratov,
Fuglede-Putnam  theorem  in the algebra of  locally measurable operators,
\textit{Proceedings Fifth Dr. George Bachman Memorial Conference, Indian Journal of  Mathematics}, 2013, {\bf 55},  13-20.

\bibitem{B} S.K. Berberian, Note on a theorem of Fuglede and Putman,
\textit{Proc. Amer. Math. Soc.}, 1959, {\bf 10}, 175-182.

\bibitem{CZ} V.I. Chilin  and B.S. Zakirov, Abstract charactaization of $EW^{*}$-algebras,
\textit{Funkts. Anal. Priiozh.}, 1991, {\bf 25}, 76-78 (Russian).

\bibitem{Dixmier} J. Dixmier, \textit{Von Neumann algebras.} North-Holland Mathematical Library, {\bf 27}. North-Holland Publishing Co., Amsterdam-New York, 1981.

\bibitem{F}  B. Fuglede, A commutativity theorem for normal operators,
\textit{Proc. Amer. Math. Soc.}, 1950, {\bf 36}, 35-40.

\bibitem{HSWY}  D. Hadwin, J. Shen,  W. Wu,  W. Yuan,  Relative commutant of an unbounded operator affiliated with a finite von Neumann algebra, \textit {J. Operator Theory}, 2016, {\bf 75}, no. 1, 209-223.

\bibitem{MCh}  M.A. Muratov and V.I. Chilin, \textit{Algebras of measurable and
locally measurable operators}, Proceedings of Institute of
Mathematics of NAS of Ukraine, 2007, {\bf 69}. (Russian).

\bibitem{Ne}  E. Nelson, Notes on non-commutative integration, {\it J. Funct.  Anal.}, 1974,
{\bf 15}, 103-116.

\bibitem{N} J. von Neumann, Approximative properties of matrices of high finite order, \textit{ Portugaliae Math.}, 1942, \textbf{3}, 1-62.

\bibitem{P}  C.R. Putnam, \textit{Commutation Properties of Hilbert Space Operators and Related Topics}, Springer-Verlag,  New York, 1967.

\bibitem{R} W. Rudin, {\it Functional analysis.} Second edition. International Series in Pure and Applied Mathematics. McGraw-Hill, Inc., New York, 1991.

\bibitem{Sak} S. Sakai, \textit{$C^{*}$--algebras and $W^{*}$--algebras},
Springer-Verlag, New York, 1971.

\bibitem{Seg}  I.E. Segal, A non-commutative extension of abstract
integration, \textit{Ann. Math.}, 1953, \textbf{57}, 401-457.

\bibitem{Tak} M. Takesaki, \textit{Theory of operator algebras I, II}, New York,
Springer-Verlag, 1979.

\bibitem{Yead} F.J. Yeadon, Convergence of measurable operators,
\textit{Proc.~Camb.~Phil.~Soc.}, 1973, \textbf{74}, 257-268.



\end{thebibliography}
\end{document}